\documentclass[11pt]{amsart}

\usepackage{amsmath}
\usepackage{amsthm}
\usepackage{amssymb}
\usepackage{mathrsfs}
\usepackage{soul}

\usepackage{xcolor}
{%
\setcounter{enumi}{0}

\begin{enumerate}}%
{\end{enumerate} }

{%
\setcounter{enumi}{0}

\begin{enumerate}}%
{\end{enumerate} }

\usepackage[colorlinks]{hyperref}
\newtheorem{theorem}{Theorem}[section]
\newtheorem{lemma}[theorem]{Lemma}
\newtheorem{proposition}[theorem]{Proposition}

\newtheorem*{remark}{Remark}

\theoremstyle{definition}
\newtheorem{definition}{Definition}

\numberwithin{equation}{section}

\allowdisplaybreaks

\newcommand*\grad{\nabla}

\newcommand*\R{\mathbb{R}}
\newcommand*\N{\mathbb{N}}
\newcommand*\rn{\mathbb{R}^n}
\newcommand*\rt{\mathbb{R}^2}

\newcommand*\omegal{\Omega_\ell}

\begin{document}

\title[Fractional Poincar\'e inequality on domains with finite ball condition]{On Fractional Poincar\'e inequality for unbounded domains  with finite ball conditions: Counter Example }

\maketitle
\centerline{\scshape Indranil Chowdhury$^1$,   Prosenjit Roy$^2$ }
\medskip
{\footnotesize
  \centerline{1. Norwegian University of Science and Technology, Norway.\ indranil.chowdhury@ntnu.no\footnote{\textit{Present Address:} PMF, University of Zagreb, Croatia. indranil.chowdhury@math.hr}}
 \centerline{2. Indian Institute of Technology,  Kanpur, India.\
 prosenjit@iitk.ac.in}
  }



\author{}
\address{}
\curraddr{}
\email{}
\thanks{}

\smallskip

\begin{abstract}  In this paper we investigate the fractional Poincar\'e inequality on unbounded domains. In the local case,  Sandeep-Mancini  [\textit{Moser-Trudinger inequality on conformal discs, Commun. Contemp. Math, 2010}] showed that in the class of simply connected domains, Poincar\'e inequality holds if and only if the domain  satisfies finite ball condition. We prove that such a result \textit{can not be true} in the `nonlocal/fractional' setting even if finite ball condition is replaced by a related stronger condition. We further provide some \textit{sufficient criterions} on domains for fractional Poincar\'e inequality to hold.  In the end, asymptotic behaviour of \textit{all eigenvalues} of fractional Dirichlet problems on long cylindrical domains is addressed.  
\end{abstract}

\maketitle



\author{}
\address{}
\curraddr{}
\email{}
\thanks{}

\smallskip

\smallskip

\date{}



\smallskip

\section{Introduction}

In recent years there has been a renewed interest in the theory of fractional Sobolev spaces. A special interest is due to the fact that these spaces play a fundamental role in the study of partial differential equations with nonlocal effects which have a wide range of physical applications, see \cite{Bucur16} and references therein.  Given an open set $\Omega \in \rn$, let us first define the quantity
 \begin{align*}
    BC(\Omega) = \sup \{r\, : \, B_r(x) \subset \Omega, x\in \Omega \} .
 \end{align*}
 We say the domain $\Omega$ satisfies the `\textit{finite ball condition}' if $BC(\Omega)< \infty$.
 
 Let us define the space $H_{\Omega}^s(\rn)$ as the  closure  of $C_c^{\infty}(\Omega)$ functions(extended by zero to whole $\mathbb{R}^n$) with respect  to the norm  
\begin{align*}
 \|u\|_{H^s(\rn)}:=  \|u\|_{L^2(\Omega)}+\left( \int_{\rn}\int_{\rn }\frac{(u(x)-u(y))^2}{|x-y|^{n+2s}}dx dy \right)^{\frac{1}{2}},
\end{align*}
where,  $C_c^{\infty}(\Omega)$ denotes the space  of smooth functions with compact support in $\Omega$. These spaces can be  viewed naturally as the fractional counterpart of $H^1_0(\Omega)$, defined to be the closure of $C_c^{\infty}(\Omega)$  with respect to the Sobolev norm $\Big(\|u\|^2_{L^{2}(\Omega)}+ \int_{\Omega} |\grad u|^2\Big)^{\frac{1}{2}}$. Particularly, $H_{\Omega}^s(\rn)$ plays a pivotal role to study the Dirichlet problems involving fractional Laplace operator $(-\Delta)^{s}$. For domains with continuous boundary, $H_{\Omega}^s(\rn)$ can also be written in particular form (c.f. \cite[Theorem 6]{valdi}):
$$ H_{\Omega}^s(\rn) = \{u \in H^s(\rn) : u= 0  \, \mbox{ a.e. in} \  \Omega^c \}, $$
where $H^s(\rn) :=  \left\{ u: \mathbb{R}^n \rightarrow \mathbb{R},  \|u\|_{H^s(\rn)} <\infty \right\}.$ We refer to \cite{brasco,kass, maz, ds, sar} and references therein for more details in this context.

 By Poincar\'e inequality in local case, we mean that the quantity
$$ \lambda_1(\Omega) :=  \inf_{ \substack{ u\in C_{c}^\infty(\Omega) \\ u\neq 0}}\frac{\int_{\Omega}|\grad u|^2}{\int_\Omega u^2} > 0.$$
\noindent Similarly,  We say that  \textit{fractional Poincar\'e} (\textbf{FP($s$)}) inequality holds for $H_{\Omega}^s(\rn)$ if,
$$P_{n,s}^2(\Omega) :=  \inf_{\substack{ u\in H_{\Omega}^s(\rn) \\ u\neq 0}}\frac{\int_{\rn}\int_{\rn }\frac{(u(x)-u(y))^2}{|x-y|^{n+2s}}dx dy}{\int_\Omega u^2} > 0.$$
It is  worth noting that fractional Sobolev spaces have many properties which  are quite similar to the properties observed in classical Sobolev spaces as well.  
Very interestingly, many results depend on the range of the fractional power $s$. We refer to \cite{lk, dy, vaha2, rupert, loss, maz} and references therein for general discussions on fractional Sobolev and Hardy's inequalities.   \smallskip



  It is well-known that for  domains with finite Lebesgue measure (in particular for bounded domains), FP($s$) inequality holds  for all $s \in (0,1)$ (for reference, see \cite{ds}, it also follows from our Theorem \ref{thm:sufficient:cond}). Also  $P_{n,s}^2(\Omega)$  corresponds to the first eigenvalue of fractional Dirichlet problem on bounded domains $\Omega$, (see, \cite[Proposition 9]{sar_val}). 
  Note that, similar to local case, finite ball condition for domains ($BC(\Omega) < \infty$) is \emph{necessary} for  FP($s$) inequality to hold true. 
   When $\Omega$ is contained  in two parallel hyperplanes (strips),  then  $P_{n,s}^2(\Omega) >0$ for all $s\in (0,1)$ (see,  \cite{karen} and  \cite{pi}). Apart from the above mentioned  class of domains, to the best of our knowledge, there are no non-trivial unbounded domain for which existence of FP($s$) inequality is discussed in literature.  Our Theorem \ref{thm:sufficient:cond} and discussions in section \ref{sec:suff_cond} provides several examples of non-trivial domains for which FP($s$) inequality holds (or does not hold).
  On the other hand in \cite{type1}, the authors discussed about the `\textit{regional fractional Poincar\'e}' inequalities for unbounded domain where for any domain $\Omega$ the best constant  $P^1_{n,s}(\Omega)$ defined as $$P_{n,s}^1(\Omega) := \displaystyle \inf_{\substack{u\in C^{\infty}_c(\Omega) \\ u\neq 0}} \frac{\int_{\Omega}\int_{\Omega }\frac{(u(x)-u(y))^2}{|x-y|^{n+2s}}dx dy}{\int_\Omega u^2}.$$ It is immediate to note that $P_{n,s}^2(\Omega)\geq P_{n,s}^1(\Omega)$. Hence, if  $P_{n,s}^1(\Omega)>0$ one find that FP($s$) inequality holds as well. Although one can not guarantee the reverse. In fact, for any bounded Lipschitz domain $\Omega$ we have $P_{n,s}^1(\Omega)=0$ for $s \in (0,\frac{1}{2}]$ (see \cite{rupert}, also \cite[Proposition 2.3]{type1}). But as mentioned above, for any bounded domain we get $P_{n,s}^2(\Omega)>0$.
 \smallskip

  Interestingly, in local case there is a direct correspondence between Poincar\'e inequality and finite ball condition. We have the following result due to Mancini-Sandeep \cite{Sandeep-Man} in  dimension $n=2$. Any higher dimension version of this result is still unknown.
\begin{proposition}[\cite{Sandeep-Man}]\label{prop:S-M}
 Let $\Omega \subset \mathbb{R}^2$ be simply connected. Then   $$BC(\Omega) < \infty \
\Leftrightarrow  \ \lambda_1(\Omega) > 0.$$
 \end{proposition}

 On the other hand, for nonlocal case, one can verify that the simply connected domain
$$\Omega =   \mathbb{R}^2\setminus \left(\mathbb{Z}\times \{(-\infty, 0] \cup [1, \infty)\} \right)$$
satisfies  finite ball condition (i.e. $BC(\Omega)<\infty$), and $P^2_{2,s}(\Omega) = P^2_{2,s}(\R^2)=0$ for $s \in (0,\frac{1}{2})$ (see, e.g. Lemma \ref{lem:extended-nec-cnd}).   The example reflects that Gagliardo seminorm in fractional case $s\in (0,\frac{1}{2})$ does not 'count' the low dimensional parts of the complement of $\Omega$. Hence, we modify the definition of $BC(\Omega)$ accordingly and define the 
following:   
\begin{align*}
    \overline{BC}(\Omega) = \sup \{r\, : \, |B_r(x)\cap \Omega^c|=0 , \, x\in \Omega \} .
 \end{align*}
\begin{definition}[\emph{Extended finite ball condition}]
We say the domain $\Omega \subset \rn$ satisfies \textit{extended finite ball condition} if $ \overline{BC}(\Omega)< \infty$. 
\end{definition}
In Lemma \ref{lem:extended-nec-cnd}, we show that $\overline{BC}(\Omega) < \infty$ is again a necessary condition on any domain for FP(s) inequality to hold  for $s\in (0, \frac{1}{2})$.  Our main aim is to show that the nonlocal analog of Proposition \ref{prop:S-M} can non hold even with the extended finite ball condition. Note that, pathological examples like above (full space like domains) are already ruled out with this new condition.


\begin{theorem}[\emph{Counter Example}] \label{thm:main}
Let $s\in (0,\frac{1}{2})$. There exists a simply connected domain  $\mathcal{D} \subset \mathbb{R}^2$ satisfying  extended finite ball condition ($\overline{BC}(\Omega) < \infty$)  for which  FP($s$) inequality  does not hold, i.e. $P_{2,s}^2(\mathcal{D})=0.$
\end{theorem}
Infinite strips are important class of unbounded domains that has finite ball condition. One can verify that (cf. Theorem \ref{thm:sufficient:cond}), if the domain is a finite union of strips then FP$(s)$ is true for all $s$. Our construction of the domain $\mathcal{D}$, in the above theorem, uses  union of infinite number of parallel strips, separated by a distance that goes to $0$ at a special rate, and another perpendicular strip joining them to keep the simply connected assumption. Theorem \ref{thm:main} assures that in the class of simply connected domains,  extended finite ball condition is not sufficient to ensure the FP($s$) inequality in the  full range of $s\in (0, \frac{1}{2})$.  In the range $s \in [\frac{1}{2}, 1)$ such a result is not available yet.
\begin{remark}
Interestingly, in Section \ref{sec:suff_cond} (as an application of Theorem \ref{thm:sufficient:cond}) we will show that for  $\mathcal{D}$, $P_{2,s}^2(\mathcal{D}) > 0$ in the regime $s \in (\frac{1}{2}, 1)$. What happens for $s= \frac{1}{2}$ is unclear to us.
\end{remark}

  Our next  result provides two sufficient criterion for  FP($s$) inequality to hold true.  We start with some definitions that  are required to formulate our next theorem.

\begin{definition}[\emph{Uniform FP($s$) Inequality}]
Let $\{\Omega_\alpha\}_{\alpha}$ be a family of sets in $\mathbb{R}^n$, where $\alpha \in \mathbb{A}$ (some indexing set). We say FP($s$)  inequality  to hold uniformly  for  $\{\Omega_\alpha\}_{\alpha}$, if  
$\displaystyle\inf_\alpha P_{n,s}^2(\Omega_\alpha) > 0.$
\end{definition}
\noindent For any $\omega\in S^{n-1}$ and $x_0 \in \rn$ we define $L_{\Omega}(x_0, \omega) := \left\{t \ |  \ x_0+t\omega \in \Omega \right\}\subset \mathbb{R}$. Here $S^{n-1}$ denotes the  unit sphere in $\mathbb{R}^n$.
\begin{definition}[\emph{LS($s$) type Domain}]\label{defn:LS}  We say $\Omega$ is of type LS($s$), if there exists a set $ \sigma \subset S^{n-1}$, of positive $n-1$ dimensional Hausdorff measure, such that one dimensional FP($s$) inequality   holds uniformly for  the  family of sets $\left\{L_\Omega(x_0,\omega) \right\}_{x_0\in \rn, \omega \in \sigma}$. \vspace{0.1cm}
\end{definition}

Note that, it is sufficient to consider the family of sets for $x_0 \in P(\omega)$ where $P(\omega)$ denotes the plane perpendicular to $\omega \in S^{n-1}$,  passing through the origin. The definition of LS(s) type domain is technical. To provide some geometric intuitions, we present several examples of LS(s) type domain in Section \ref{sec:suff_cond}.

\begin{theorem}[\emph{Sufficient Criterion}]\label{thm:sufficient:cond}
Let $\Omega \subset \rn$ be a domain and $s \in (0,1)$. Then $P_{n,s}^2(\Omega) > 0$ if $\Omega$ satisfies one of the following criterion:
\begin{itemize}
    \item[(i)] There exist  $R>0$ and $c>0$ such that $|\Omega^c \cap B(x,R)|>c$ for each $x \in \Omega$.
    \item[(ii)] $\Omega$ is a \textrm{LS(s) type domain}.
\end{itemize}
\end{theorem}
\smallskip
\noindent We believe that condition (i) of Theorem \ref{thm:sufficient:cond} is known to the experts, although we provide the proof for completeness. The  main tool  to  prove part (ii) of Theorem \ref{thm:sufficient:cond} is the clever use of change  of variable  type formula due to Loss-Sloane \cite{loss} which  effectively  reduces  the  problem in to one dimension setting. In Section \ref{sec:suff_cond} we present several  non-trivial examples of domains to discuss the sufficient conditions (Theorem \ref{thm:sufficient:cond}) in details.  

 
 \smallskip
Our next aim is to analyze the asymptotic behaviour of eigenvalues for fractional Laplacian on the class of domain of type $\omegal := (-\ell,\ell)^{m}\times \omega$, where  $\omega \subset \mathbb{R}^{n-m}$ be bounded open set and $n>m\in \mathbb{N}$. In this context, let us consider the following eigenvalue problem:
\begin{equation}\label{frac eigen value}
    \begin{cases}
       (-\Delta)^su_\ell=\lambda(\omegal)\;u_\ell\;\;\text{ in }\Omega_\ell,\\
       u_\ell=0\;\;\;\;\text{ in }\Omega_\ell^c=\rn\setminus\Omega_\ell.
    \end{cases}
\end{equation}
For detail discussion on the spectrum of fractional eigenvalue problem,  we refer to \cite{fr,sar_val}. We establish the following theorem regarding the asymptotic behaviour of the $k$-th eigenvalue of the above problem as $\ell \rightarrow \infty,$
\begin{theorem}[\emph{{Asymptotics of the k-th Eigenvalue}}]\label{eigenvalue conv}
It holds that for $0<s<1$ and $k\in\N$
$$P^2_{n-m,s}(\omega)\leq\lambda_{k}(\omegal)\leq P^2_{n-m,s}(\omega)+A\ell^{-s}$$
where $A$ is a constant independent of $\ell$ and $\lambda_k(\Omega_\ell)$ denotes the $k$-th eigenvalue of \eqref{frac eigen value}.
\end{theorem}
\smallskip
\noindent  For the case when $k=1$, the above theorem  characterises the best Poincar\'e constant for the strip like domain $\R^{m} \times \omega$ and this is established in \cite{type1, musina}.  For the local analogue of Theorem \ref{eigenvalue conv} (that is for second order elliptic operator in divergence form with Dirichlet boundary condition), we refer to  \cite{arn}. Independently,  study of problems  on $\Omega_\ell$ for large $\ell$ is  carried out by several authors in the last two decades.
For more literature on this subject we refer to  \cite{allaire,chipot1,delpino, cis, crs, zube,donato, gues} for the result considering local operator and \cite{musina, pi, karen} for nonlocal operators. we refer to \cite{dy-k} for related result, regarding the study of spectral gap of fractional Laplace like operator on rectangular domain.
 \smallskip

 The article is organized as follows. In section \ref{sec2} we provide some preparatory lemmas and well known results. In Section \ref{sec:domain}, we construct the domain $\mathcal{D}$ as in Theorem \ref{thm:main} and present the prove of Theorem \ref{thm:main}.  In Section \ref{sec:suff_cond}, we prove Theorem \ref{thm:sufficient:cond} and as an application of it, we present some examples of domains for which FP($s$) inequality is true. Finally, in Section \ref{final section}, we present the proof of Theorem \ref{eigenvalue conv}.

\section{Preliminary and technical Lemmas}\label{sec2}
 We  introduce  some notations that will be followed uniformly through  out this article. For any Lebesgue measurable subset $E \subset \mathbb{R}^n $, the measure will be  denoted by $|E|$. A ball of radius $r$ and centre at $x$ will be denoted by by $B_r(x)$. For real number $x$, $[x]$ denotes the greatest integer less than or equal to $x$. In this section we introduce  some known results and some technical lemma, that  will be useful for the proof of our result.  
  For $u \in H^s_\Omega(\rn)$, we will denote its Gagliardo semi norm by
 $$[u]_{s,\Omega,\rn} = \int_{\rn}\int_{\rn } \frac{|u(x)-u(y)|^2}{|x-y|^{n+2s}}dxdy.$$
 
 \begin{lemma}
  \label{bfc}
  If $\Omega \subset \R^n$ does not satisfy finite ball condition, then $P_{n,s}^2(\Omega) = 0$.
 \end{lemma}
 \begin{proof}
   Fix $0\neq U   \in C_c^\infty(B_1(0))$ and define $ \lambda := \frac{[U]_{s,B_1(0),\rn}}{\int_{B_1(0)}|U(x)|^2 dx}$. Clearly, $\lambda< \infty$.  Domain  not satisfying finite  ball condition implies that for any $R > 0$ (large) there exist $x_R \in \Omega $ such that  $B_R(x_R) \subset \Omega$. Shifting  the coordinate system  to $x_R$ and defining $v(x) = U(\frac{x}{R}) $ it is easy  to see that
\begin{align*}
 \hspace*{4.1cm}P_{n,s}^2(\Omega) \leq  \frac{[v]_{s,\Omega,\rn}}{\int_\Omega v^2} = R^{-2s}\lambda \xrightarrow[R\to \infty]{} 0. \hspace{5.1cm} \qedhere
\end{align*}
 \end{proof}

Next, we establish that for $s\in (0,\frac{1}{2})$ even extended finite ball condition is necessary for FP(s) inequality  to hold. 
\begin{lemma}\label{lem:extended-nec-cnd}
Let $s\in (0, \frac{1}{2})$. If $\Omega \subset \R^n$ does not satisfy extended finite ball condition (i.e. $\overline{BC}(\Omega)=\infty$), then $P_{n,s}^2(\Omega) = 0$.
\end{lemma}
\begin{proof}
The assumption $\overline{BC}(\Omega)=\infty$ implies that there exists sequences $\{x_k\}_k\subset \Omega$ and $\{R_k\}_k\subset \R^+$ such that $R_k\to \infty$ and 
\begin{align*}
    |B_{R_k}(x_k) \cap \Omega^c| =0.
\end{align*}
Hence by De morgan's law we have $|B_{R_k}(x_k) \cap \Omega|= |B_{R_k}(x_k) |$. Consider the function 
\begin{align*}
   \psi_k = \begin{cases}
       1 \quad & \text{in} \ B_{R_k}(x_k) \cap \Omega, \\
       0 \quad &  \text{in} \ \big(B_{R_k}(x_k) \cap \Omega\big)^c,
    \end{cases}
 \qquad   \tilde \psi_k= \begin{cases}
       1 \quad & \text{in} \ B_{R_k}(x_k), \\
       0 \quad &  \text{in} \ \big(B_{R_k}(x_k))^c,
    \end{cases}
\end{align*}
and note that $\psi_k \in H^s_{\Omega}(\rn)$ (see the calculation below). As $|B_{R_k}(x_k) \cap \Omega|= |B_{R_k}(x_k) |$, we find
\begin{align*}
    \int_{\rn\times \rn} \frac{(\psi_k(x)-\psi_k(y))^2}{|x-y|^{n+2s}} dx dy = \int_{\rn\times \rn} \frac{(\tilde \psi_k(x)-\tilde \psi_k(y))^2}{|x-y|^{n+2s}} dx dy \\ 
    = 2 \int_{B_{R_k}(x_k)} \int_{(B_{R_k}(x_k))^c} \frac{1}{|x-y|^{n+2s}} dx dy = 2 R_k^{n-2s} P_s(B_1(0))
\end{align*}
where $P_s(B_1(0)) := \int_{B_{1}(0)}\int_{(B_{1}(0))^c} \frac{1}{|x-y|^{n+2s}} dx dy$, and the last equality follows by change of variable. By \cite[Corollary 4.4]{brasco-Lindgren-Parini} we find  $P_s(B_1(0))<\infty$ for $s\in (0, \frac{1}{2})$.  Therefore we have
\begin{align*}
    P^2_{n,s}(\Omega) \leq \frac{\int_{\rn\times \rn} \frac{(\psi_k(x)-\psi_k(y))^2}{|x-y|^{n+2s}} dx dy}{\int_{\rn} |\psi_n|^2 dx} =\frac{2 R_k^{n-2s}P_s(B_1(0))}{|B_{R_k}(x_k)|} \leq K R_k^{-2s}
    \xrightarrow[k\to \infty]{} 0.
\end{align*}
This completes the proof. 
\end{proof}

The following lemma provides a sufficient condition on family of domains on real line for FP($s$) inequality to hold.
\begin{lemma}\label{ap}
Consider  $\displaystyle \Omega= \cup_{j\in \mathbb{Z}}\, (a_j, b_j) \subset \mathbb{R}$, where $(a_j,b_j)$ are mutually disjoint. Also let $\displaystyle M = \max_{j}|a_j -b_j| < \infty$ (that is $BC(\Omega) < M$). Then, for $s \in (\frac{1}{2},1)$, one has for some constant $C >0,$
$$P_{1,s}^2(\Omega) \geq CM^{-2s}.$$
\end{lemma}
 \begin{proof}
For $u \in C_c^\infty(\Omega)$, we have
\begin{equation}
    \label{ffgg}
    \int_\R \int_\R \frac{|u(x) - u(y)|^2}{|x-y|^{1+2s}} dxdy
    \geq \sum_{j=-\infty}^\infty \int_{a_j}^{b_j}\int_{a_j}^{b_j} \frac{|u(x) - u(y)|^2}{|x-y|^{1+2s}} dxdy.
\end{equation}
 For $s\in \big(\frac{1}{2},1\big)$, noting that $P^1_{1,s}\big((0,1)\big)<\infty$ and by taking suitable scaling, translation and change of variable we find
 \begin{align*}
 \int_{a_j}^{b_j}\int_{a_j}^{b_j} \frac{|u(x) - u(y)|^2}{|x-y|^{1+2s}} dxdy \geq P^1_{1,s}\big((0,1)\big) |a_j-b_j|^{-2s} \int_{a_j}^{b_j} u^2 dx,
 \end{align*}
 for every $j \in \mathbb{Z}$. As $|a_j-b_j|\leq M$ we thus have
 \begin{align*}
 \int_\R\int_\R \frac{|u(x) - u(y)|^2}{|x-y|^{1+2s}} dxdy  \geq P^1_{1,s}\big((0,1)\big) M^{-2s}\sum_{j=-\infty}^\infty \int_{a_j}^{b_j}u^2(x)dx = P^1_{1,s}\big((0,1)\big) M^{-2s}\int_\Omega u^2.
 \end{align*}
 This finishes the proof of the lemma.
\end{proof}

\begin{lemma}\emph{[An inequality]}\label{inequal}
For $m\in (0,1),$ and  $a, b \in \mathbb{R}$,  \  $||a|^m -|b|^m|\leq |a-b|^m.$
 \end{lemma}
 \begin{proof}
 It suffices to show the result for $a,b\geq 0$. To prove the inequality, consider the function $f:\R^+ \rightarrow \R $ as $f(x)= (x+c)^m - c^m -x^m$, where $c\geq 0$ is a fixed constant. Then for any $x>0$, $$f'(x) = p \Big(\frac{1}{(x+c)^{1-p}} - \frac{1}{x^{1-p}}\Big)\leq 0.$$
  Therefore, $f(x)$ is monotonically decreasing and as $f(0)=0$, for any fixed $c\geq 0$ and $x\geq 0$ we get
 $$(x+c)^m - c^m -x^m\leq 0.$$
 Whenever, $a\geq b$ the result follows by taking $c=b$ and $x=a-b$. Whereas, for $a<b$ the result follows by taking $c=a$ and $x=b-a$.
 \end{proof}

\begin{lemma} \label{lem:dblint:id}
Let $a, b \in \mathbb{R}$ such that $|a| < |b|$, then there exist a constant depending on $n$ and $s \in (0,1)$ such that
\begin{align*}
\int_{|a|}^{|b|} \int_{-\infty}^{\infty} \frac{1}{|y|^{2+2s}}dy_2dy_1 = C(s) \Big( \frac{1}{|a|^{2s}} - \frac{1}{|b|^{2s}}\Big) .
\end{align*}
\end{lemma}
\begin{proof}
Let $y=(y_1,y_2)$, then by change of variable formula we see by choosing $y_2 = y_1 \tan \theta$ that
\begin{align*}
\int_{|a|}^{|b|} \int_{-\infty}^{\infty} \frac{dy}{|y|^{2+2s}}  & = \int_{|a|}^{|b|} \bigg(\int_{-\frac{\pi}{2}}^{\frac{\pi}{2}} \frac{y_1 \sec^2 \theta \, d\theta}{y_1^{2+2s}(1+ \tan^2 \theta)^{2+2s}} \bigg) \, d y_1 \\
& = \bigg( \int_{|a|}^{|b|} y_1^{-2s-1} \, dy_1\bigg) \bigg( \int_{-\frac{\pi}{2}}^{\frac{\pi}{2}} \frac{d \theta}{(\sec \theta)^{2s}}\bigg) = C(s) \bigg( \frac{1}{|a|^{2s}} - \frac{1}{|b|^{2s}} \bigg)
\end{align*}
where, $2sC(s) =  \int_{-\frac{\pi}{2}}^{\frac{\pi}{2}} {(\cos \theta)^{2s}}d \theta$. It completes the proof.
\end{proof}

The following lemma will be used several times in the proof of our main result.

\begin{lemma}\label{maineq}Let  $ 0 <  q_1  <  q_2$, $M,N > 0$ and $s \in (0,\frac{1}{2})$. Define $\mathcal{B}_{M,N} :=  (0,M)\times (0,N)$ and $\mathcal{S}_{q_1,q_2} = {(-q_2,-q_1)\times (-\infty, \infty) }$, then
\begin{equation*}
\int_{\mathcal{B}_{M,N}} \int_{\mathcal{S}_{q_1,q_2}}\frac{dydx}{|x-y|^{2+2s}} =
C(s)N\left[   (q_1 +M)^{1-2s} -q_1^{1-2s} -(q_2+M)^{1-2s} +q_2^{1-2s} \right]
\end{equation*}
where $C(s) > 0$ is some constant depending on $s$.
\end{lemma}
\begin{proof}
For fixed $x \in {B}_{N,N},$ introduce the change  of variable by $z = y-x.$ Then $z \in (-q_2-x_1,-q_1-x_1)\times (-\infty, \infty)$. Then  the required integral becomes,
$$\int_{\mathcal{B}_{M,N}} \int_{\mathcal{S}_{q_1,q_2}}\frac{dydx}{|x-y|^{2+2s}} = \int_{\mathcal{B}_{M,N}} \left(  \int_{(-q_2-x_1,-q_1-x_1)\times (-\infty, \infty)}\frac{dz}{|z|^{2+2s}}\right) dx. $$
Now  from the previous lemma we obtain,
\begin{multline*}
\int_{\mathcal{B}_{M,N}} \int_{\mathcal{S}_{q_1,q_2}}\frac{dydx}{|x-y|^{2+2s}} =  C(s)\int_{\mathcal{B}_{M,N}}\left\{  \frac{1}{(q_1+x_1)^{2s}} -\frac{1}{(q_2+x_1)^{2s}} \right\}dx \\
=C(s)N\int_{0}^M \left\{  \frac{1}{(q_1+x_1)^{2s}} -\frac{1}{(q_2+x_1)^{2s}} \right\}dx_1 \\
=C(s)N\left[   (q_1 +M)^{1-2s} -q_1^{1-2s} -(q_2+M)^{1-2s} +q_2^{1-2s} \right].
\end{multline*}
This completes the proof of the lemma.
\end{proof}

\section{Domain not having fractional Poincar\'e} \label{sec:domain}

\noindent  Let $x= (x_1,x_2) \in \mathbb{R}^2$ and define a decreasing sequence $\{s_j\}_{j\in \N}$  with  the following property:
\begin{equation}\label{dec}
\sum_{m=0}^\infty s_{m}^{1-2s} <\infty.
\end{equation}
Precise form of the sequence $s_j$ will be given later. We will construct the domain with the countable union of infinite strips. The definition of the domain is the following:
\begin{align*}
\mathcal{C}_k &:= \big( a_k, \,1+ a_k \big) \times (-\infty, \infty) ; \quad k\geq 0 \quad \mbox{where,} \quad a_k:= k+ \sum_{j=0}^k s_j , \ s_0 =0 \\
\mathcal{S}_k &:= \big( a_k-s_k, \,a_k \big) \times (-\infty, \infty) ; \quad k\geq 1.
\end{align*}
$\mathcal{S}_k$ denotes the strip between $\mathcal{C}_{k-1}$ and $\mathcal{C}_k$. We denote the strips similarly on the left hand side of $Y$-axis as well,
\begin{align*}
\mathcal{C}_k &:= \big( a_k, \, a_k+1 \big) \times (-\infty, \infty) ; \quad k\leq -1 \quad \mbox{where,} \quad a_k:= k - \sum_{j=k+1}^0 s_{-j} , \\
\mathcal{S}_k &:= \big( a_k -s_{|k|}, \,a_k \big) \times (-\infty, \infty) ; \quad k\leq -1.
\end{align*}
For convention we denote $S_0= \emptyset $. Let us now define a domain $\Omega_0$ by
\begin{align}\label{domain_eg}
\Omega_0 = \displaystyle \bigcup_{k=-\infty}^{\infty} \mathcal{C}_k .
\end{align}
We note  that $\Omega_0$ is  symmetric (reflection)  about  the line $x_1 =0$. \smallskip

Finally, we denote $ D = (-\infty, \infty) \times (-2,-1)$ and take the following simply connected domain to proof Theorem \ref{thm:main}:
\begin{align*}
\mathcal{D} := \Omega_0 \cup D.
\end{align*}

Now we present the proof of Theorem \ref{thm:main}. We remark that the value of arbitrary constant will be denoted by $C, C(s)$ or $K$  in the proof and it may change from line to line.
\vspace{0.5cm}

\noindent \textbf{Proof of the Theorem \ref{thm:main}.}
Clearly $\mathcal{D}$ is a simply connected domain and $\overline{BC}(\mathcal{D})<\infty$.  Proof of the theorem consists of different steps. We will prove  the theorem by constructing a sequence of function $\{\psi_k\}_k$ for each $k \in \N$ and then claiming $P_{s,\mathcal{D}}(\psi_k)\rightarrow 0$ as $k \rightarrow \infty$, where
\begin{align*}
    P_{s,\mathcal{D}}(\psi) =  \frac{[\psi]_{s,\mathcal{D},\mathbb{R}^2}}{\int_{\mathcal{D}} \psi^2(x) \, dx}.
\end{align*}
For $j \in \mathbb{Z},\ k_0 \in \mathbb{Z}^+$ define $C_j^{k_0} := \{ (x_1, x_2)  \in \mathcal{C}_j \ | \ x_2 \in (0, k_0) \}$ and the function
\begin{equation}
    \label{psi}
\displaystyle\psi_{k,k_0}(x)= \left\{
       \begin{array}{ll}
      1   \quad &\mbox{for}\   x\in \cup_{j=0}^k C_j^{k_0},  \\
         0    \quad &\mbox{for} \  x\in \mathbb{R}^2 \setminus\cup_{j=0}^k C_j^{k_0}.    
       \end{array}
  \right.
\end{equation}
Note that ${support}\{\psi_{k,k_0}\} \subset \Omega_0 \subset \mathcal{D}$. Without any loss of generality we will simply denote  $ \psi_{k,k_0}$ by $\psi$ for rest of the argument.
\smallskip

\noindent \textbf{Step 1:}\, We write
\begin{align*}
& \int_{\rt}\int_{\rt }  \frac{\big(\psi(x+y)-\psi(x)\big)^2}{|y|^{2+2s}} \, dy \, dx \\ = & \int_{x\in\Omega_0 }\int_{x+y\in \Omega_0 }\frac{\big(\psi(x+y)-\psi(x)\big)^2}{|y|^{2+2s}} \, dy \, dx + 2\int_{x\in\Omega_0 }\int_{x+y\in \Omega_0^c }\frac{\big(\psi(x+y)-\psi(x)\big)^2}{|y|^{2+2s}} \, dy \, dx \\
&  + \int_{x\in\Omega_0^c }\int_{x+y\in \Omega_0^c }\frac{\big(\psi(x+y)-\psi(x)\big)^2}{|y|^{2+2s}} \, dy \, dx \, \\
= & \, \, \mathbb{I}_1 + 2  \, \mathbb{I}_2,
\end{align*}
where $\mathbb{I}_1$ and $\mathbb{I}_2$ denotes the first and second integral in the previous expression. The third integral becomes zero as $\psi=0$ on $\Omega_0^c$. We will estimate $\mathbb{I}_1$ and $\mathbb{I}_2$ separately.\smallskip

\noindent \textbf{ Step 2 (Estimate of $\mathbb{I}_2$)}:
 \ From \eqref{psi} we have
\begin{align*}
\mathbb{I}_2 & = \int_{x\in \Omega_0} \int_{x+y \in \Omega_0^c} \frac{ \psi^2(x)}{|y|^{2+2s}} \, dy \, dx    = \sum_{m=-\infty}^{\infty}\bigg[  \int_{\mathcal{C}_m} \psi^2(x) \Big( \sum_{j=-\infty}^{\infty} \int_{x+y \in \mathcal{S}_j} \frac{dy}{|y|^{2+2s}}\Big) dx\bigg]\\
&= \sum_{m=0}^{k} \sum_{j=-\infty}^{\infty} \int_{x \in \mathcal{C}_m} \psi^2(x) \bigg( \int_{x+y \in \mathcal{S}_j} \frac{dy}{|y|^{2+2s}}\bigg) dx  
\\&= \sum_{m=0}^{k} \sum_{j=-\infty}^{[m/2]} \mathbb{I}_{2,m,j} + \sum_{m=0}^{k} \sum^{j=\infty}_{\substack{[m/2]+1 \\ j\neq m , m+1}} \mathbb{I}_{2,m,j} + \sum_{m=0}^k (\mathbb{I}_{2,m,m} + \mathbb{I}_{2,m,m+1})\\
& := \mathbb{J}_1 +\mathbb{J}_2+  \mathbb{J}_3.
\end{align*}
In the above expression  $$\mathbb{I}_{2,m,j} = \int_{\mathcal{C}_m} \psi^2(x) \Big( \int_{x+y \in \mathcal{S}_j} \frac{dy}{|y|^{2+2s}}\Big) dx. $$
For any fixed $x = (x_1, x_2)\in \rt$ and $j
\in \mathbb{Z} \setminus \{0\}$,  $x+y
\in \mathcal{S}_j $ if and only if $y \in \mathcal{S}_j- \{x\}$. Then by Lemma \ref{lem:dblint:id}, we have
\begin{align*}
    \int_{x+y \in \mathcal{S}_j} \frac{dy}{|y|^{2+2s}} = C \left|\frac{1}{|a_j  -x_1|^{2s}} - \frac{1}{|a_j -s_{|j|} -x_1|^{2s}} \right| := \mathcal{G}_{j}(x).
\end{align*}
Therefore,  by \eqref{psi} the term $\mathbb{I}_{2,m,j}$ can be written as
\begin{align} \label{esti1}
\mathbb{I}_{2,m,j} & =  \int_{x \in \mathcal{C}_m} \psi^2(x) \mathcal{G}_{j}(x) \, dx  =  \int_{x \in \mathcal{C}_m^{k_0}}\mathcal{G}_{j}(x)dx.
\end{align}
\noindent \underline{Estimate for $\mathbb{J}_3$:}\, It is enough to estimate the first term  $\sum_{m=0}^k \mathbb{I}_{2,m,m}$ of $\mathbb{J}_3$, as the estimate for the other term follows similarly.
 Now, by making a simple modification of Lemma \ref{maineq} in accordance to apply for $\mathbb{I}_{2,m,m}$, we have that  
\begin{equation*}
\label{im1}
| \mathbb{I}_{2,m,m}|=  C(s)k_0 |1 - (1+s_{m})^{1-2s}+ s_{m}^{1-2s}| \leq C(s) k_0 s_{m}^{1-2s}.
\end{equation*}
One can use Taylor's expansion or Lemma \ref{inequal} to obtain the second last inequality in the previous line. Using \eqref{esti1} with $j=m$, we have
\begin{equation}
    \label{j_3} \mathbb{J}_3 = \sum_{m=0}^k( \mathbb{I}_{2,m,m} + \mathbb{I}_{2,m,m+1}) \leq C(s) k_0\sum_{m=0}^k s_m^{1-2s}.
\end{equation}
\noindent \underline{Estimate for $\mathbb{J}_1$:}\
We note, given $2s<1$, there exist $P_0 \in \mathbb{N}$  such that $\frac{1}{P_0+1}< 2s \leq \frac{1}{P_0}$. First, denote $A_j := |a_j-s_{|j|} -x_1|^{2s} $ and $B_j:=|a_j  -x_1|^{2s}$.  Multiply the  numerator and denominator of $\mathcal{G}_{j}(x)$ by $\sum_{\ell =0}^{P_0-1} A_j^{P_0-1-\ell}B_j^{\ell}$ to get
\begin{eqnarray}\label{ijm}
\mathcal{G}_{j}(x) =  \frac{|A_j^{P_0} - B_j^{P_0}|}{A_j B_j \Big( \displaystyle  \sum_{\ell =0}^{P_0-1} A_j^{P_0-1-\ell}B_j^{\ell} \Big)} .
\end{eqnarray}

From the definition of $\mathcal{C}_m$ and $ \mathcal{S}_j$, we see that whenever $x\in \mathcal{C}_m$,
\begin{equation}
 \label{lowerbd_qt1}
|a_j-x_1| ;\,  |a_j -s_{|j|} -x_1| \geq \left\{
       \begin{array}{ll}
      |m-j| -1    \quad &\mbox{for}\  j\geq m+2,  \\
         |m-j|    \quad &\mbox{for} \  j \leq m-1.    
       \end{array}
  \right.
\end{equation}
As $0<2sP_0<1$, from Lemma \ref{inequal} we obtain
\begin{equation}\label{dd}
|A_j^{P_0} - B_j^{P_0} |\leq s_{|j|}^{2sP_0}.
\end{equation}
 Therefore from \eqref{ijm}, \eqref{lowerbd_qt1} and \eqref{dd}, we get
\begin{equation}
    \label{psid}
 \displaystyle
|\mathcal{G}_{j}| \leq \left\{
       \begin{array}{ll}
      \frac{s_{|j|}^{2sP_0}}{(|m-j| -1)^{2s(P_0+1)}}   \quad &\mbox{for}\    j \geq m+2,  \\
          \frac{s_{|j|}^{2sP_0}}{|m-j|^{2s(P_0+1)}}    \quad &\mbox{for} \  j\leq m-1.
       \end{array}
  \right.
\end{equation}
Therefore from \eqref{esti1}, we have
\begin{equation*}
    \label{estt3}
 \displaystyle
\mathbb{I}_{2,j,m} \leq \left\{
       \begin{array}{ll}
      \frac{k_0s_{|j|}^{2sP_0}}{(|m-j| -1)^{2s(P_0+1)}}   \quad &\mbox{for}\    j \geq m+2,  \\
          \frac{k_0s_{|j|}^{2sP_0}}{|m-j|^{2s(P_0+1)}}    \quad &\mbox{for} \  j\leq m-1.
       \end{array}
  \right.
\end{equation*}

 Since $s_j \rightarrow 0$, we can find a constant $C$ such that $s_{|j|} \leq C, \forall j.$ Using this, we have
\begin{align}\label{j11}
\mathbb{J}_1 = \sum_{m=0}^k\sum_{j=-\infty}^{[\frac{m}{2}]} \mathbb{I}_{2,j,m}  
& \leq Ck_0\sum_{m=0}^k \sum_{j=-\infty}^{[\frac{m}{2}]} \frac{1}{|m-j|^{2s(P_0+1)}} \leq Ck_0 \sum_{m=0}^k\sum_{j=[\frac{m}{2}]}^{\infty} \frac{1}{j^{2s(P_0+1)}}\notag\\& \leq Ck_0\sum_{m=0}^k \int_{[\frac{m}{2}]}^{\infty} \frac{dz}{z^{2s(P_0+1)}} \leq C(s) k_0\, \sum_{m=0}^km^{1 - 2s(P_0+1)}\notag\\&= Ck_0k^{2 - 2s(P_0+1)}.
\end{align}
\noindent \underline{Estimate for $\mathbb{J}_2$:}\
Using \eqref{psid} and the decreasing property of $\{s_j\}_j$, we obtain,
\begin{equation*}
    \label{fgg}
    \sum_{\substack{j=[\frac{m}{2}]+1, \\ j \neq m, m+1}}^{\infty} \mathcal{G}_{j} \leq s_{[\frac{m}{2}]}^{2sP_0}\sum_{j=[\frac{m}{2}]+1}^{m-1} \frac{1}{|j-m|^{2s(P_0+1)}}  + s_{m+2}^{2sP_0}\sum_{j=m+2}^{\infty} \frac{1}{|j-m|^{2s(P_0+1)}}.
\end{equation*}
Using \eqref{dec} and  $\sum_{j=1}^{\infty} \frac{1}{j^{2s(P_0+1)}} < K$ for some constant $K>0$, we obtain
\begin{align}\label{dddc}
 \sum_{\substack{j=[\frac{m}{2}]+1,\\ j \neq m, m+1}}^{\infty} \mathcal{G}_{j} \leq 2K s_{[\frac{m}{2}]}^{2sP_0}.
\end{align}
 From \eqref{esti1} and \eqref{dddc}, we get
\begin{align}
   \label{poi}
\mathbb{J}_2 &=  \sum_{m=0}^k \bigg(\int_{\mathcal{C}_m^{k_0}} \sum_{\substack{j=[\frac{m}{2}],\\ j \neq m, m+1}}^{\infty} \mathcal{G}_{j} \bigg) \leq Kk_0 \sum_{m=0}^k s_{[\frac{m}{2}]}^{2sP_0}.
\end{align}
Combining \eqref{j_3}, \eqref{j11} and \eqref{poi},
\begin{equation}
    \label{lp}
    \mathbb{I}_2 \leq Ck_0 \left( k^{2-2s(P_0+1)} + \sum_{m=0}^k s_{m}^{1-2s} +  \sum_{m=0}^k s_{[\frac{m}{2}]}^{2sP_0}\right).
\end{equation}

\noindent \textbf{Step 3 (Estimate of $\mathbb{I}_1$):} \, We will now estimate the term $\mathbb{I}_1$.  After changing the variable, we write
\begin{align}\label{i1}
\mathbb{I}_1 &= \sum_{j,m=-\infty}^{\infty} \int_{\mathcal{C}_m\times \mathcal{C}_j} \frac{\big(\psi(x)-\psi(y)\big)^2}{|x-y|^{2+2s}} \, dy \, dx \nonumber\\
&=\bigg[2 \bigg( \sum_{j=-\infty}^{-1}\sum_{m=0}^k+  \sum_{j=k+1}^{\infty}\sum_{m=0}^k\bigg) + \sum_{j=0}^{k}\sum_{m=0}^k \bigg] \int_{x\in \mathcal{C}_m}\int_{y\in \mathcal{C}_j}\frac{\big(\psi(x)-\psi(y)\big)^2}{|x-y|^{2+2s}} \,  dy \, dx\nonumber \\
& := 2(A_1 + A_2) + A_3 . \notag
\end{align}
The estimate of $A_1$ and $A_2$ will be similar and follow the similar line of argument as in Step 2, whereas for the term $A_3$ the arguments will be quite different and delicate. \smallskip

\noindent \underline{Estimate of $A_1$ and $A_2$:} \,
We start with the term $A_1$,
\begin{align*}
A_1 & = \sum_{j=-\infty}^{-1}\sum_{m=0}^k \int_{x\in \mathcal{C}_m}\int_{y\in \mathcal{C}_j}\frac{\psi^2(x)}{|x-y|^{2+2s}} \,  dy \, dx \\&\leq C  \sum_{j=-\infty}^{-1}\sum_{m=0}^k  \int_{x\in \mathcal{C}_m^{k_0}}\bigg| \frac{1}{|a_j-x_1|^{2s}} - \frac{1}{|a_j +1 -x_1|^{2s}}\bigg| dx.
\end{align*}
When $(j,m) = (-1,0)$, we can estimate
$$\int_{x\in \mathcal{C}_0}\int_{y\in \mathcal{C}_{-1}}\frac{\psi^2(x)}{|x-y|^{2+2s}} \, dx \, dy  \leq Ck_0,$$
after using Lemma \ref{maineq}, with $q_1=0, q_2= 1, M=1$ and $N=k_0$.
Next we use the similar argument as done for the term in Estimate of $\mathbb{J}_1$ and the fact that
$$|a_j +1 -x_1|, |a_j -x_1|\geq m-j-1 \quad \mbox{for} \quad (j,m) \in \{\cdots,-1\}\times \{0,\cdots,k\} \setminus\{(-1,0)\}$$
to obtain
\begin{equation}\label{a1}
 A_1 \leq   C k_0 + C \sum_{j=-\infty}^{-1}\sum_{\substack{m=0\\(j,m)\neq (-1,0)}}^k \frac{ |\mathcal{C}_m^{k_0}|}{(m-j-1)^{2s(P_0+1)}}\leq C k_0 \left( k^{2-2s(P_0+1)} +1 \right).
\end{equation}
The estimate of $A_2$ is exactly similar, in fact one can also see this by observing that  $A_2 \leq A_1$ (as the integrand in $A_1$ is point wise less  than the integrand in $A_2$, after a change of variable.)
Therefore,
\begin{equation}
    \label{a2w}
    A_1+ A_2 \leq C k_0 \left( k^{2-2s(P_0+1)} +1 \right).
\end{equation}
\noindent \underline{Estimate of $A_3$:} \,
Finally, we estimate the term $A_3$. Note that for each $j,m$, one has
$$\int_{\mathcal{C}_j}\int_{\mathcal{C}_m}  \frac{\big(\psi(x)-\psi(y)\big)^2}{|x-y|^{2+2s}} \ dxdy \leq \int_{\mathcal{C}_j}\int_{\mathcal{C}_j}  \frac{\big(\psi(x)-\psi(y)\big)^2}{|x-y|^{2+2s}} \ dxdy.$$
To see this, we first make a change of variable by performing a reflection $T$ with respect to $x_2=\frac{a_m-(a_j+1)}{2}$, which sends $C_m$ to $C_j$(without loss of generality assuming $m>j$). One can easily verify that  $|x-T(y)|\geq |x-y|$ for any $x,y \in C_j$.
Therefore, above estimate holds by noting $\psi(T(y))= \psi(y)$, by the definition of $\psi$.  Hence,
\begin{align*}
A_3 = \sum_{{j,m =0}}^{k} \int_{\mathcal{C}_j}\int_{\mathcal{C}_m}  \frac{\big(\psi(x)-\psi(y)\big)^2}{|x-y|^{2+2s}} \ dxdy \leq (k+1) \sum_{j=0}^k\int_{\mathcal{C}_j}\int_{\mathcal{C}_j}  \frac{\big(\psi(x)-\psi(y)\big)^2}{|x-y|^{2+2s}} \ dxdy.
\end{align*}
 Further, we note that we can replace the domain of integration $\mathcal{C}_j \times \mathcal{C}_j$ by $\mathcal{C}_0\times \mathcal{C}_0$, which follows by simple change of variables. Therefore from \eqref{psi},
\begin{align*}
P:= \int_{\mathcal{C}_0\times \mathcal{C}_0} \frac{\big(\psi(x)-\psi(y)\big)^2}{|x-y|^{2+2s}} \ dxdy = 2\int_{ \mathcal{C}_0^{ k_0}}\int_{\big ( {\mathcal{C}_0 \setminus \mathcal{C}_0^{ k_0}}\big )} \frac{dydx}{|x-y|^{2+2s}}.
\end{align*}
Using $\mathcal{C}_0 \setminus \mathcal{C}_0^{ k_0} \subset (-\infty,\infty)\times(k_0, \infty),$ we get
$$P \leq 2\int_{ \mathcal{C}_0^{ k_0}}  \int_{(-\infty,\infty)\times(k_0, \infty)}\frac{dydx}{|x-y|^{2+2s}}=  2\int_{ \mathcal{C}_0^{ k_0}}\left(  \int_{-\infty}^\infty\int_{k_0-x_2}^{\infty}\frac{1}{|z|^{2+2s}}dz\right)dx.$$
Now using Lemma \ref{lem:dblint:id}, we get we have for some positive constant $C$ ,
\begin{equation*}
    \label{c_0}
    P \leq C \int_{ \mathcal{C}_0^{ k_0}} \frac{dx}{|k_0-x_2|^{2s}}= Ck_0^{1-2s}.
\end{equation*}
Therefore, $A_3 \leq C k^2k_0^{1-2s}.$
We finally have from  \eqref{a1}, \eqref{a2w} and from the previous inequality, that
\begin{align}\label{ll}
\mathbb{I}_1  \leq C k_0 k^{2-2s(P_0+1)} + C k^2k_0^{1-2s}.
\end{align}

\noindent \textbf{Step 4 (Final steps):} \, Note that
\begin{align}\label{de}
\int_{\mathcal{D}}\psi^2(x) \, dx = \sum_{j=0}^k \int_{\mathcal{C}_j}\psi^2(x) \, dx = k k_0.
\end{align}
 Therefore from \eqref{lp}, \eqref{ll}, \eqref{de} and using $1-2s\leq 2s P_0$ we have
\begin{align*}
P_{s,\mathcal{D}}(\psi) =\frac{\mathbb{I}_1 + \mathbb{I}_2 }{k k_0} \leq C\left(k^{1-2s(P_0+1)} +  kk_0^{-2s}+k^{-1}\sum_{m=0}^k s_{m}^{1-2s} + k^{-1}\sum_{m=0}^k s_{[\frac{{m}}{2}]}^{1-2s}\right).
\end{align*}
Now choose $k_0 = k^A$, where $2sA  > 1$ and the choice of $s_m$ such that $\sum_{m=0}^\infty s_{m}^{1-2s}< \infty$, then clearly the right hand side of the above expression tends to zero as $k\rightarrow \infty.$ This completes the proof of the theorem.
\qed

\section{On Sufficient Conditions} \label{sec:suff_cond}

In this section, we prove Theorem \ref{thm:sufficient:cond} and discuss on the sufficient conditions in details. Before  we present the proof,  let us first  recall the key identity by Loss- Sloane \cite{loss} which applies to prove the second part of Theorem \ref{thm:sufficient:cond}.

\begin{lemma}\emph{\textbf{[Loss-Sloane]}}\label{loss}
Let $\Omega \subset \mathbb{R}^n$,  then for $u \in C_c^\infty(\Omega)$,
$$I_{n, s, \Omega} [u] =\frac{1}{2} \int_{\omega \in S^{n-1}} \int_{\{x : x\cdot \omega = 0\}}\int_{x+ s\omega \in \Omega,\\  x+t\omega \in \Omega } \frac{(u(x+s\omega)- u(x+t\omega))^2}{|s-t|^{1+ 2s}}dtdsd\mu(x)d\omega$$
where $\mu$ denotes the $n-1$ dimensional Lebesgue measure on the plane $x \cdot \omega = 0.$
\end{lemma}\smallskip

\noindent \textbf{ Proof of Theorem \ref{thm:sufficient:cond}.} \ \textit{Part one:} We start with the right hand side of the inequality, for any $u \in C_c^{\infty}(\Omega)$
\begin{align*}
\int_{\rn\times \rn} \frac{\big(u(x)-u(y)\big)^2}{|x-y|^{n+2s}} \, dx \, dy \geq \int_{x\in \rn} \int_{y \in \Omega^c \cap B(x,R)} \frac{\big(u(x)-u(y)\big)^2}{|x-y|^{n+2s}} \, dx \, dy.
\end{align*}
As $u=0$ on $\Omega^c$ and $\frac{1}{|x-y|^{n+2s}}\geq \frac{1}{R^{n+2s}}$ for each $y \in \Omega^c \cap B(x,R)$, we finally have
\begin{align*}
 \int_{x\in \rn} \int_{y \in \Omega^c \cap B(x,R)} \frac{\big(u(x)-u(y)\big)^2}{|x-y|^{n+2s}} \, dx \, dy
 & \geq \frac{1}{R^{n+2s}} |\Omega^c \cap B(x,R)|\int_{\rn} u^2(x) \, dx\\ &\geq \frac{c}{R^{n+2s}} \int_{\Omega} u^2(x) \, dx.
 \end{align*}
Hence the result follows.  \smallskip

\textit{Second Part:} \  
Take $\Omega = \mathbb{R}^n$ in Lemma \ref{loss}.  Then
$$2I_{n,s, \mathbb{R}^n}[u]\geq  \int_{\omega \in \sigma} \int_{\{x : x\cdot \omega = 0\}} \left( \int_{s, t \in \mathbb{R}} \frac{(u(x+s\omega)- u(x+t\omega))^2}{|s-t|^{1+ 2s}}dtds\right)d\mu(x)d\omega.$$
Notice that since we have assumed LS(s) property on the domain, this implies  that for  each fixed $ \omega \in \omega, x \in P(\omega)$ there exist a constant $C >0$, independent of $\sigma$ and $P(\omega)$ such that
$$ \int_{s, t \in \mathbb{R}} \frac{(u(x+s\omega)- u(x+t\omega))^2}{|s-t|^{1+ 2s}}dtds \geq C \int_{\mathbb{R}} u^2(x+s\omega)ds.$$
Plugging  the above two inequalities  together  we obtain,
\begin{align*}
&I_{n,s, \mathbb{R}^n}[u]\geq  \frac{C}{2}\int_{\omega \in \sigma} \int_{\{x : x\cdot \omega = 0\}} \left( \int_{\mathbb{R}} u^2(x+s\omega)ds \right)d\mu(x)d\omega
 \\=&   \frac{C}{2}\int_{\omega \in \sigma} \left(\int_{\{x : x\cdot \omega = 0\}}  \int_{\mathbb{R}} u^2(x+s\omega)ds d\mu(x)\right)d\omega  =   \frac{C}{2}\int_{\omega \in \sigma} \int_{\mathbb{R}^n}u^2dx d\omega  = \frac{C|\sigma|}{2}\int_\Omega u^2dx.
 \end{align*}
 This finishes the proof of the theorem. \qed \bigskip
 
 As an application of Theorem \ref{thm:sufficient:cond}, we now present some examples of unbounded domains for which fractional Poincar\'e inequality is true:\bigskip
 
\textbf{Example 1: Finite union of infinite strips.} By infinite strip we mean the region contained in between two parallel hyperplanes. It is very easy to verify that the criteria in Theorem \ref{thm:sufficient:cond} (i) holds here. Therefore, fractional Poincare inequality hold for all $s\in (0,1)$. \smallskip

\textbf{Example 2: \ $\mathcal{D}$ as in Theorem \ref{thm:main}.} \
 For $s \in (\frac{1}{2},1)$ there is an \textit{easy geometric characterisation for  any domain $\Omega$ to satisfy LS$(s)$ condition}. A domain $\Omega$ satisfies LS($s$) condition  if and only if
$$\sup_{x_0\in \rn, \omega \in \sigma} BC(L_\Omega(x_0,\omega))  < \infty,$$
where  the sets $\left\{L_\Omega(x_0,\omega) \right\}_{x_0\in \rn, \omega \in \sigma}$ is as in Definition \ref{defn:LS}.
This follows as an immediate application of Lemma \ref{bfc} and Lemma \ref{ap}.  From this it clear that $\mathcal{D}$ for $s > \frac{1}{2}$
 satisfies LS($s$) condition and hence the FP($s$) inequality holds.\smallskip

\ \textbf{Example 3: Infinite number of parallel infinite strips.} If $\Omega$ is union of infinite number of parallel infinite strips, each one is of width $1$, with the property that distance between any two consecutive strips is bounded below by a strictly positive number. Then FP($s$) is true for all $s\in (0,1)$ both as an application of condition (i) and (ii) in Theorem \ref{thm:sufficient:cond}.  \smallskip
 
  \textbf{Example 4: Infinite Strips with decreasing width.} \   First let us consider the following one dimensional set.
 
 $$\Omega_1 =   (0,1) \cup (2, 2+\frac{1}{2})\cup (3, 3+\frac{1}{3}) \cup  (3+\frac{2}{3},  3+\frac{2}{3}+ \frac{1}{4}) \cup \cdots$$
 Basically length of $n$-th interval is $\frac{1}{n}$ and the distance between $n$ and $n+1$-th interval is also $\frac{1}{n}$.  Then define
 $$\Omega = \Omega_1 \cup \left( -\Omega_1 \right).$$ One can check that $P_{1,s}^2(\Omega) > 0$ as an application of condition (i) of Theorem \ref{thm:sufficient:cond}. Also  if we consider  $D = \Omega \times (-\infty, \infty) \subset \mathbb{R}^2$. Then, it is easy to check  that $P_{1,s}^2(D) > 0$ as an application of condition (i) of Theorem \ref{thm:sufficient:cond}. This example  serves as an example where the width of the strips ( of both the domain and its complement) goes to $0$.
 \smallskip
 
\textbf{Example 5: Concentric balls.}
The following domain satisfies the first criterion for all $s\in (0,1)$, but not a domain of LS(s) type.
$$
   \Omega:=\bigcup_{k=1}^{\infty}B_{2k}(0)\setminus B_{2k-1}(0).
$$
\smallskip
 
\textbf{Example 6: Domain with holes at $\mathbb{Z}\times \mathbb{Z}$ coordinates.} For $n,m \in \mathbb{Z}$, let $B_{r}((n,m))$ denotes the ball centered at $(n,m)$ and radius  $r$ for any $r>0$ small enough. It is easy to check that the following domain satisfies the condition (i) of Theorem \ref{thm:sufficient:cond}:
$$\displaystyle\Omega := \mathbb{R}^2\setminus \left(\bigcup_{n,m \in \mathbb{Z}} B_{r}((n,m)) \right).$$

\section{Proof of Theorem \ref{eigenvalue conv}}\label{final section}
 As mentioned in the introduction, the proof of the theorem for $k=1$ is done in \cite{type1}, but we will present some details of the proof of the sake of completeness. First, let us start with some  preliminary results that will be useful to prove Theorem \ref{eigenvalue conv}. For $\Omega \subset \rn$
\begin{equation*}\label{frac eigen valuef}
    \begin{cases}
       (-\Delta)^s u=\lambda(\Omega)u\;\;\text{ in }\Omega,\\
       u=0\;\;\;\;\text{ in }\Omega^c=\rn \setminus\Omega.
    \end{cases}
\end{equation*}
It is well known (see, \cite{sar_val})  that the set of eigenvalues for the above problem are discrete and tends to infinity. The first eigenvalue is simple and strictly positive. If $\lambda_k(\Omega)$ denotes the $k$-th eigenvalue and $u^k$ denotes the corresponding eigenfunction,  then
\begin{eqnarray}
\label{hev}
\lambda_k(\Omega) = \inf_{\substack{v\in H^s_{\Omega_{\ell}}(\rn)\setminus\{0\}\\ v\perp u^1, \cdots , u^{k-1}}}\frac{\int_{\rn}\int_{\rn}\frac{|v(x)-v(y)|^2}{|x-y|^{n+2s}}dxdy}{\int_{\Omega}v^2(x)\;dx}.
\end{eqnarray}
In the above expression $v\perp u^i$ means that $\int_{\Omega} vu^i = 0,$ for $i = 1, 2, \cdots , k-1.$ Also, for any $x \in \mathbb{R}^d$, we will use the following:
\begin{equation}\label{xcb}
    P_{n,s}^2(\Omega) =   P_{n,s}^2(x + \Omega).
\end{equation}

Now  we present the proof of the theorem \ref{eigenvalue conv}.
\bigskip

\noindent \textbf{ Proof of Theorem \ref{eigenvalue conv}.} First we consider the  case when $k=1$.   In \cite[Theorem 1.4]{type1}, it is established that  $P^2_{n-m,s}(\omega)=P^2_{n,s}(\mathbb{R}^m\times \omega).$ Now, the first part of the required inequality, that is,
$P^2_{n-m,s}(\omega)\leq P_{n,s}^2(\Omega_\ell)$
follows from the domain monotonicity property  of $P^2_{n,s}$ (If $\mathcal{D} \subset \Omega_2$, then $P^2_{n,s}(\Omega_2) \leq P^2_{n,s}(\mathcal{D})$).
The second part of the required inequality follows  following the similar argument as in the proof of Proposition 3.1 in \cite{type1}.

\smallskip

 Now we consider the case when $k=2$. We divide the domain $\omegal$ in the $x_1$ direction with equal Lebesgue measure as follows:
\begin{align*}
Q_{1,\ell} & =\bigg(-\ell,-\frac{\ell}{3}\bigg)\times (-\ell,\ell)^{m-1}\times\omega,\;Q_{12,\ell}=\bigg(-\frac{\ell}{3},\frac{\ell}{3}\bigg)\times (-\ell,\ell)^{m-1}\times\omega,\;\\
Q_{2,\ell} & =\bigg(\frac{\ell}{3},\ell\bigg)\times (-\ell,\ell)^{m-1}\times\omega.
\end{align*}
We denote by $\lambda_{1}(Q_{i,\ell})$  the first eigenvalue of the problem  \eqref{hev}, where $\Omega$ is replaced by $Q_{i,\ell},\;i=1,2$ and $v_{i,\ell}$ is the corresponding normalized first eigenfunctions respectively.  Since $\Omega_{\frac{\ell}{3}}\subset Q_{i,\ell}\subset\omegal$(we have identified  $\Omega_{\frac{\ell}{3}}$ with its appropriate translate), then it holds by using \eqref{xcb} that
\begin{equation}\label{ccv}
P^2_{n-m,s}(\omega) \leq\lambda_{1}(\omegal)\leq \lambda_{1}(Q_{i,\ell})\leq\lambda_{1}(\Omega_{\frac{\ell}{3}})\leq P^2_{n-m,s}(\omega) +\frac{3^s\;C}{\ell^s}.
\end{equation}
In the last step we have used  Theorem 1.2 of \cite{type1} where the case $k=1$ is considered. \smallskip

Define the function $$\psi_\ell := c_1 v_{1,\ell} + c_2v_{2,\ell}$$ where $c_1, c_2 \in \mathbb{R}$.
 We can choose both $c_1,\; c_2$ to be non zero,  such that  
 \begin{align} \label{perp_eigenfn}
 \int_{\Omega_\ell} \psi_\ell u_\ell = c_1\int_{\Omega_\ell} v_{1,\ell} u_\ell +c_2 \int_{\Omega_\ell} v_{2,\ell} u_\ell  = 0,
 \end{align}
 where $u_\ell$ denotes the first eigenfunction of the problem \eqref{frac eigen value}.   Now we calculate the fractional semi norm of the function $\psi_\ell$,
 \begin{align*}
     & \int_{\rn}\int_{\rn}\frac{|\psi_\ell(x)-\psi_\ell(y)|^2}{|x-y|^{n+2s}}dxdy  \\
     = & \int_{\rn}\int_{\rn}\frac{|c_1\;(v_{1,\ell}(x)-v_{1,\ell}(y))+c_2\;(v_{2,\ell}(x)-v_{2,\ell}(y))|^2}{|x-y|^{n+2s}}\;dxdy.
 \end{align*}
 With out any loss of generality we can assume the $u_\ell$ is the normalized eigenfunction, that is
 \begin{equation}
     \label{normali}
     \int_{\Omega_\ell} u_\ell^2 =1.
 \end{equation}
 Notice that $v_{1, \ell}$ and $v_{2,\ell}$ has disjoint supports, therefore we have
 \begin{align*}
       \int_{\rn}\int_{\rn}\frac{|\psi_\ell(x)-\psi_\ell(y)|^2}{|x-y|^{n+2s}}dxdy &=c_1^2\lambda_1(Q_{1,\ell}) +c_2^2\lambda_1(Q_{2,\ell}) \\ &+ 2c_1c_2\int_{\Omega_\ell}\int_{\Omega_\ell}\frac{(v_{1,\ell}(x)-v_{1,\ell}(y))(v_{2,\ell}(x)-v_{2,\ell}(y))}{|x-y|^{n+2s}}\;dxdy.
 \end{align*}
 Using \eqref{xcb}, we obtain $\lambda_1(Q_{1,\ell}) =\lambda_1(Q_{2,\ell})$. We can further simplify  the second integral above  to get
  \begin{align*}
       \int_{\rn}\int_{\rn}\frac{|\psi_\ell(x)-\psi_\ell(y)|^2}{|x-y|^{n+2s}}dxdy = \left( c_1^2 +c_2^2\right) \lambda_1(Q_{1,\ell}) - 2c_1c_2\int_{Q_{2,\ell}}\int_{Q_{1,\ell}}\frac{v_{1,\ell}(x)v_{2,\ell}(y)}{|x-y|^{n+2s}}\;dxdy.
 \end{align*}
 Using Young's inequality,
  \begin{align}\label{kdl}
       & \int_{\rn}\int_{\rn}\frac{|\psi_\ell(x)-\psi_\ell(y)|^2}{|x-y|^{n+2s}}dxdy \leq \left( c_1^2 +c_2^2\right) \lambda_1(Q_{1,\ell})\nonumber \\
       & \hspace{3cm}+c_1^2\int_{Q_{2,\ell}}\int_{Q_{1,\ell}}\frac{|v_{1,\ell}(x)|^2}{|x-y|^{n+2s}}\;dxdy+ c_2^2\int_{Q_{2,\ell}}\int_{Q_{1,\ell}}\frac{|v_{2,\ell}(y)|^2}{|x-y|^{n+2s}}\;dxdy.
       \end{align}
 We will only present the estimate for the term $\int_{Q_{2,\ell}}\int_{Q_{1,\ell}}\frac{|v_{1,\ell}(x)|^2}{|x-y|^{n+2s}}\;dxdy$. The estimate for the other integral follows similarly.  Using $|x-y| \geq \frac{2\ell}{3}$ for $x \in Q_{1,\ell}$ and $y \in Q_{2,\ell}$ and \eqref{normali} we derive that
 \begin{equation}\label{estima}
 \displaystyle \left|\int_{Q_{2,\ell}}\int_{Q_{1,\ell}}\frac{|v_{1,\ell}(x)|^2}{|x-y|^{n+2s}}\;dxdy \right| \leq \frac{C}{\ell^{n+2s}}|Q_{2,\ell}| = \frac{C}{\ell^{n-m+2s}}.
 \end{equation}
 Therefore from \eqref{kdl} and \eqref{estima}, we get
 \begin{equation}
   \label{estimm}
       \int_{\rn}\int_{\rn}\frac{|\psi_\ell(x)-\psi_\ell(y)|^2}{|x-y|^{n+2s}}dxdy
      \leq \left( c_1^2 +c_2^2\right) \lambda_1(Q_{1,\ell}) + \frac{C}{\ell^{n-m+2s}}.
 \end{equation}
 Now we use \eqref{normali}  to get  
 \begin{equation}\label{estimmma}
 \displaystyle\int_{\omegal}\psi_{\ell}^2(x)\;dx=c_1^2\displaystyle\int_{Q_{1,\ell}}v_{1,\ell}^2(x)\;dx+c_2^2\displaystyle\int_{Q_{2,\ell}}v_{2,\ell}^2(x)\;dx=c_1^2+c_2^2.
 \end{equation}
By the identity \eqref{hev} and noting the fact \eqref{perp_eigenfn}, we find
\begin{align*}
    \lambda_2(\Omega_{\ell}) \leq \frac{\int_{\rn}\int_{\rn} \frac{|\psi_{\ell}(x)-\Psi_{\ell}(y)|^2 }{|x-y|^{n+2s}} dx dy}{\int_{\rn} \psi_{\ell}^2(x) dx}.
\end{align*}
 Therefore from \eqref{ccv}, \eqref{estimm} and \eqref{estimmma}, we have
 \begin{equation*}
     \lambda_2(\Omega_\ell) \leq  \lambda_1(Q_{1,\ell}) + \frac{C}{\ell^{n-m+2s}} \leq P_{n-m,s}^2(\omega) +\frac{C}{\ell^{s}}+ \frac{C}{\ell^{n-m+2s}}.
 \end{equation*}
 The result then follows after using $\lambda_1(\omegal) < \lambda_2(\omegal)$ and \eqref{ccv}.\smallskip

For the case of general $k$, we have to split  the domain $\omegal$ into $2k-1$ subdomains in $x_1$ direction with equal Lebesgue measure and proceed similarly as done above. \qed \bigskip

\textbf{Acknowledgement:}  \  The first author was supported by the ERCIM
``Alain Bensoussan" Fellowship programme at NTNU, Norway. Parts of this work  is carried when the first author was visiting IIT, Kanpur.  Research work of second author is funded by Matrix grant  (MTR/2019/000585)   and Inspire grant  (IFA14-MA43) of Department of Science and Technology (DST). We would like to thank Prof. B. Dyda for useful comments and suggestions.

\end{document}